\documentclass[11pt]{article}
\usepackage[letterpaper,margin=0.9in]{geometry}
\usepackage{amsmath,amssymb,amsthm,mathtools}
\usepackage{booktabs,tabularx,microtype}
\usepackage[colorlinks=true,linkcolor=blue,citecolor=blue,urlcolor=blue]{hyperref}
\usepackage{enumitem}
\usepackage{mathrsfs}
\newtheorem{theorem}{Theorem}[section]
\newtheorem{corollary}[theorem]{Corollary}
\newtheorem{lemma}[theorem]{Lemma}
\newtheorem{proposition}[theorem]{Proposition}
\theoremstyle{definition}
\newtheorem*{remark}{Remark}
\theoremstyle{plain}
\newcommand{\R}{\mathbb R}
\newcommand{\C}{\mathbb C}
\newcommand{\T}{\mathbb T}
\newcommand{\Sch}{\mathscr S}
\newcommand{\Z}{\mathbb Z}
\newcommand{\N}{\mathbb N}

\renewcommand{\Re}{\operatorname{Re}}
\renewcommand{\Im}{\operatorname{Im}}
\makeatletter
\NewCommandCopy\@@pmod\pmod
\DeclareRobustCommand{\pmod}{\@ifstar\@pmods\@@pmod}
\def\@pmods#1{\mkern4mu({\operator@font mod}\mkern 6mu#1)}
\makeatother
\allowdisplaybreaks[4]
\AtBeginDocument{
 \setlength{\abovedisplayskip}{9pt plus 2pt minus 4pt}
 \setlength{\belowdisplayskip}{9pt plus 2pt minus 4pt}
}
\title{The optimal van der Corput property for squares}
\author{
Steve Fan\thanks{Steve.Fan@uga.edu} \and
Andrew Lott\thanks{Andrew.Lott@uga.edu}
}
\date{October 11, 2026}
\begin{document}
\maketitle

\begin{center}
University of Georgia
\end{center}

\begin{abstract}
We prove a sharp quantitative form of the van der Corput property for squares. For all sufficiently large $N$, we construct a real trigonometric polynomial $P(\theta)=a_0+\sum_{1\le n\le\sqrt N}a_n\cos(2\pi n^2\theta)$ satisfying
\begin{equation}\label{eq:abstract-vdc}
 P(\theta)\ge0\quad(\theta\in\R/\Z),\qquad P(0)=1,
\end{equation}
with 
\[0\le a_0\le\exp(-c\sqrt{\log N\log\log N}),\]
where $c>0$ is an absolute constant. As a consequence, if $A\subseteq\N\cap[1,N]$ has no nonzero square difference, then 
\[|A|\le N\exp(-c\sqrt{\log N\log\log N}),\]
improving upon a recent result of Green and Sawhney. Our construction builds on previous work of Slijep{\v{c}}evi{\'c}, Slijep{\v{c}}evi{\'c}--Nin{\v{c}}evi{\'c}, and Fan--Lott, and the proof of the upper bound is elementary and self-contained. We also show that every real trigonometric polynomial of the above form satisfying \eqref{eq:abstract-vdc} has 
\[a_0\ge\exp(-C\sqrt{\log N\log\log N})\]
for an absolute constant $C>0$. Thus the van der Corput property cannot be used to obtain a power-saving bound for the Furstenberg--S\'ark\"ozy theorem.
\end{abstract}

\tableofcontents

\section{Introduction}
For a set $F\subseteq \mathbb Z$ and $N>1$, write
\[
D(F,N):= \max\bigl\{|A|:\ A\subseteq [N]\ \text{and}\ (A-A)\cap F\subseteq\{0\}\bigr\},
\]
where $[N]:=\N\cap[1,N]$. Here, we refer to $F$ as the set of \emph{forbidden differences}. 

The Furstenberg--S\'ark\"ozy theorem asserts that $D(\{n^2:n\in\mathbb N\},N)=o(N)$, answering a question of Lov\'asz.
This was proved independently by Furstenberg~\cite{Fur77} using ergodic theory and by S\'ark\"ozy~\cite{sarkozy-i,Sar78} using the Hardy--Littlewood circle method and the density increment method of Roth; S\'ark\"ozy's proof gives the quantitative bound
\[
D(\{n^2:n\in\mathbb N\},N) \ll N(\log\log N)^{2/3}(\log N)^{-1/3}.
\]
Pintz--Steiger--Szemer\'edi~\cite{pintz-steiger-szemeredi} introduced a more elaborate double iteration strategy to obtain a bound of the form
\[
D(\{n^2:n\in\mathbb N\},N) \ll N(\log N)^{-c_1\log\log\log\log N}
\]
with $c_1=1/12$, and Bloom--Maynard~\cite{bloom-maynard} later strengthened this to
\[
D(\{n^2:n\in\mathbb N\},N) \ll N(\log N)^{-c_2\log\log\log N}
\]
where $c_2>0$ is an absolute constant.
Green and Sawhney~\cite{GreenSawhney2024FSv1,GreenSawhney2025FSv2} developed a new ``arithmetic level-$d$'' inequality (building on hypercontractivity estimates of Keller--Lifshitz--Marcus~\cite{KLM23} and Keevash--Lifshitz--Long--Minzer~\cite{kllm}) and pushed the density increment strategy to obtain the following quasipolynomial bound
\begin{equation}
D(\{n^2:n\in\mathbb N\},N) \ll N\exp\bigl(-c_3(\log N)^{1/2}\bigr),
\end{equation}
for some constant $c_3>0$.
They also construct an example which they interpret as ruling out bounds
stronger than $N\exp(-c(\log N)^{3/4})$ by the density increment method,
``at least in anything like the form in the present paper (and in previous works)''
\cite[Section~9]{GreenSawhney2025FSv2}.

\begin{samepage}
Another approach for bounding $D(F,N)$ is the quantitative van der Corput
method. For a positive integer $N$ and a small parameter $\delta>0$,
one seeks a real trigonometric polynomial
\[
 P(\theta)=a_0+\sum_{n=1}^N a_n\cos(2\pi n\theta)
\]
with the following properties:
\begin{enumerate}[label=(\roman*)]
 \item $P(\theta)\ge0$ for every $\theta\in\T$;
 \item $P(0)=1$;
 \item $a_n=0$ whenever $1\le n\le N$ and $n\notin F\cup(-F)$;
 \item $0\le a_0\le\delta$.
\end{enumerate}
\end{samepage}
We call $P$ a \emph{van der Corput witness} for $F$. These properties imply $D(F,N)\ll N\delta$.  
This implication goes back to Kamae--Mendès France~\cite{KM78}
and Ruzsa~\cite{RuzsaConnections}; we give the proof for squares
in Appendix~\ref{sec:corollary}.
Montgomery~\cite{MontgomeryTenLectures} gives a useful account
of the method. A recent breakthrough of Green \cite{Green23} gives a demonstration of the strength
of this method: by proving a power-saving form of the van der Corput property for the
shifted primes, he obtained the first power-saving bound for Sárközy's theorem for
differences of the form \(p-1\) with $p$ prime.

For squares in $[N]$, Slijep\v{c}evi\'c~\cite{Slijepcevic2010Squares} constructed such polynomials with constant coefficient $O((\log N)^{-1/3})$.
Slijep\v{c}evi\'c's construction for squares averages over a carefully chosen
finite collection of dilations so that their real major-arc contributions
combine to give a uniform lower bound.
Nin\v{c}evi\'c and Slijep\v{c}evi\'c~\cite{NincevicSlijepcevic2013} subsequently used averaging over dilations for odd integer polynomials of degree $k\ge3$, obtaining constant coefficient $O((\log N)^{-1/k})$.

More recently, Fan and Lott \cite{FanLott} used the quantitative van der Corput method to prove that if $A\subseteq[N]$ and $A-A$ does not contain any positive integer of the form $x^2+y^2$ with $x,y\in\Z$, then
$|A|\ll_\epsilon N^{7/8+\epsilon}$
for every fixed $\epsilon>0$. The constructed cosine polynomials also have nonnegative coefficients. The key idea there is to average over dilations in the style of Nin\v{c}evi\'c--Slijep\v{c}evi\'c and then carefully select smooth weights so that the complete sum and oscillatory integral lie in the same sector in the complex plane defined by $|\arg z|\le\pi/4$, resulting in the positivity of the real part of their product. 

Building on these ideas, we prove the following quantitative van der Corput property for squares using cosine polynomials with signed coefficients.

\begin{theorem}\label{thm:main}
For every sufficiently large integer $N$ there is a real trigonometric polynomial
\[
 P(\theta)=a_0+\sum_{1\le n\le\sqrt N}a_n\cos(2\pi n^2\theta)
\]
such that $P\ge0$, $P(0)=1$, and
\[
 0\le a_0\le\exp(-c\sqrt{\log N\log\log N}),
\]
where $c>0$ is some constant. The coefficients $a_n$ are allowed to have either sign.
\end{theorem}

\begin{corollary}\label{cor:main}
There is some absolute constant $c>0$ such that every set $A\subseteq\{1,\ldots,N\}$ with no nonzero square differences satisfies
\[
 |A|\le N\exp(-c\sqrt{\log N\log\log N})
\]
whenever $N$ is sufficiently large.
\end{corollary}

The use of signed coefficients in Theorem~\ref{thm:main} is essential.
For $F=\{n^2:n\in\N\}$, a result of Ruzsa, as reported by
Slijep\v{c}evi\'c~\cite[Section~2]{Slijepcevic2010Squares}, shows
that any polynomial satisfying (i)--(iii) with nonnegative
coefficients $a_n$ must have $a_0\gg1/\log N$.
Thus the bound in Theorem~\ref{thm:main} requires some negative
coefficients. On the other hand, the
following matching lower bound shows that Theorem~\ref{thm:main} is essentially best possible.

\begin{theorem}\label{thm:lower-bound}
There is an absolute constant $C>0$ such that, for every sufficiently
large integer $N$, every real trigonometric polynomial
\[
 P(\theta)=a_0+\sum_{1\le n\le\sqrt N}a_n\cos(2\pi n^2\theta)
\]
satisfying $P\ge0$ and $P(0)=1$ has
\[
 a_0\ge\exp\bigl(-C\sqrt{\log N\log\log N}\bigr).
\]
The coefficients $a_n$ may have either sign.
\end{theorem}

Consequently, Theorem~\ref{thm:main} is sharp up to the constant
in the exponent among the cosine polynomials satisfying (i)--(iii).

Using the quantitative van der Corput method with a different construction, it is also possible to prove that if a subset $A\subseteq [N]$ lacks cube differences, then $|A|\ll N^{1-c}$ for some constant $c>0$. This will be the main result of a forthcoming paper. The constant $c>0$ can be made explicit and reasonable. 

We also mention two further applications of the quantitative van der Corput method due separately to Kalogirou, Lott, Magyar, and Singha Roy~\cite{OddPowers2026} on sums of an even number of equal odd powers of positive integers, and to Younis~\cite{YounisBracket2026} on bracket quadratics. Both obtain power-saving bounds.

For lower bounds, Ruzsa~\cite{RuzsaSquares1984} constructed sets with no nonzero square difference of size $\gg N^{0.733077\ldots}$. Lewko~\cite{Lewko2015} and, independently, Beigel and Gasarch~\cite{BeigelGasarch2008} improved the exponent to $\frac12(1+\log 12/\log 205)=0.7334117\ldots$. Krachun~\cite{Krachun2026} broke the natural $3/4$ barrier of this approach, obtaining the exponent $0.7527964558\ldots$. Several further AI-generated constructions have also been claimed~\cite{Jones2026a,Naslund2026,Jones2026b}. (See~\cite{ForbiddenDifferencesNotes2025} for more background on lower-bound constructions.)

Generalizing the Furstenberg--S\'{a}rk\"{o}zy problem on squares, one may also consider sets $A\subseteq[N]$ whose difference sets avoid nonzero integers of the form $h(n)$ with $n\in\N$, where $h\in\Z[x]$ is a fixed intersective polynomial of degree at least two. Due to a simple congruence obstruction, this is the maximal class of polynomials for which a Furstenberg--Sárközy-type conclusion could hold.  The state of the art is due to Adajar, Agrawal, Choudhuri, Chuah, Fan, Hegde, Lott, Nandakumar, and Ponagandla \cite{IntersectiveGroup}, who combined the method of Green and Sawhney with novel exponential sum estimates in the style of Rice \cite{Ric19} to show that for every fixed $\mu\in(0,1/2)$,
\[
D(h(\N),N)\ll_{h,\mu}N\exp\bigl(-c_{h}(\log N)^\mu\bigr).
\]
Doyle and Rice \cite{DoyleRice2021} obtained similar bounds for a broad class of intersective multivariable polynomials $P\in\Z[x_1,\dots,x_\ell]$ of degree at least two.

\begin{remark}
A prominent AI company has circulated a draft claiming the power-saving bound
$D(\{n^2:n\in\N\},N)\ll N^{1-c}$ for some absolute constant $c>0$.
The initial repository commit containing this draft is timestamped
6 October 2026 at 21:58:50 UTC.\footnote{See the
\href{https://github.com/openai/math/commit/adc7f1241b42e322a6451854ab7e4b4c146bf78a}{initial repository commit}.}
The draft appears to be AI-written and largely unreadable.
The repository also claims to provide a Lean certificate for the result.\footnote{See the
\href{https://github.com/openai/math/blob/adc7f1241b42e322a6451854ab7e4b4c146bf78a/lean/docs/182.md}{stated scope of the Lean formalization}.}
The first version of our manuscript, containing complete proofs of
our results, was submitted to arXiv on
6 October 2026 at 05:02:35 UTC, exactly
16 hours, 56 minutes, and 15 seconds before that recorded commit.\footnote{See the submission history of
\href{https://arxiv.org/abs/2610.07765}{arXiv:2610.07765}.}
\end{remark}

\section{Acknowledgments}
The second author would like to thank Imre Ruzsa for helpful conversations. We would also like to thank Ben Green for his helpful suggestions and encouragement.

\section{AI disclosure}
The authors used a large language model (LLM) as a sounding board while developing the proof. In particular, they used the LLM to experiment with van der Corput constructions for squares using signed coefficients, and the LLM had access to the authors' paper on sums of two squares \cite{FanLott}. This led to the definition of $\lambda_v$, which was then combined with the ideas in \cite{FanLott} to develop the argument in this manuscript. The proof of Theorem \ref{thm:lower-bound} was discovered during later conversations with the LLM. The LLM was also used for copyediting, LaTeX assistance, and literature search, but the manuscript was prepared and verified by the authors, who take responsibility for the contents of the proof.

\section{Notation and conventions}
\label{sec:notation}

We collect here the notation and conventions used throughout the paper.

\begin{itemize}[leftmargin=*]

\item For a positive integer \(N\), we write
\[
  [N]=\{1,2,\dots,N\}.
\]
If \(A\) is a set of integers, then
\[
  A-A=\{a-a':a,a'\in A\}.
\]
We write \(1_A\) for the indicator function of \(A\).

\item Let $\T=\R/\Z$. For \(t\in\R\), we put $e(t)=e^{2\pi i t}$.

\item All divisors and divisor indices are positive unless stated otherwise.
The notation $(a,q)$ denotes the greatest common divisor, and $\mu$ is the
M\"obius function.

\item Given a nonzero integer $n$ and a positive real number $y$, we say that $n$ is \textit{$y$-smooth} if every prime factor of $n$ is at most $y$ and \textit{$y$-rough} if every prime factor of $n$ is strictly greater than $y$.

\item We write \(\Sch(\R)\) for the Schwartz class of complex-valued functions.  So \(f\in\Sch(\R)\) means that \(f\) is smooth and that for every pair of
nonnegative integers \(\alpha,\gamma\in\Z_{\ge0}\),
\[
  \sup_{x\in\R}
  \left|x^\alpha f^{(\gamma)}(x)\right|
  <\infty.
\]
Equivalently, \(f\) and all of its derivatives decay faster than any power.

\item For \(f\in\Sch(\R)\), we have the Fourier transform:
\[
  \widehat f(\xi)
  =
  \int_{\R} f(x)e(-x\xi)\,dx,
  \qquad \xi\in\R.
\]
All integrals over \(\R\) are with respect to Lebesgue measure.

\item We also have the Poisson summation formula: if $f\in \Sch(\R)$, then 
\[
  \sum_{n\in\Z} f(n)
  =
  \sum_{\ell\in\Z} \widehat f(\ell).
\]
We will also use the scaled form
\[
  \sum_{n\in\Z} f\!\left(\frac nH\right)
  =
  H\sum_{\ell\in\Z}\widehat f(H\ell)
\]
for all $H>0$, which follows by applying Poisson summation to the function \(x\mapsto f(x/H)\).

\item We will also use Poisson summation after splitting into residue classes.
If \(f\in\Sch(\R)\), \(q\in\mathbb{N}\),
and \(r\in\Z/q\Z\), then
\begin{equation}\label{eq:residue-poisson}
\sum_{\substack{n\in\Z\\ n\equiv r\pmod* q}} f(n)
  =
  q^{-1}\sum_{\ell\in\Z}
  e\!\left(\frac{r\ell}{q}\right)
  \widehat f\!\left(\frac{\ell}{q}\right),  
\end{equation}
which follows by applying Poisson summation to the function $x \mapsto f(qx+r)$.

\item On \(\Z/M\Z\), we use the discrete Fourier transform
\[
  \widehat f(\xi)
  =
  \sum_{x\pmod* M} f(x)e\!\left(\frac{x\xi}{M}\right),
  \qquad \xi\in\Z/M\Z,
\]
where $f:\Z/M\Z\to\C$ is any function. With this convention,
Parseval's identity is
\[
  \frac1M\sum_{\xi\pmod* M}|\widehat f(\xi)|^2
  =
  \sum_{x\pmod* M}|f(x)|^2.
\]

\item We use the Landau--Vinogradov notation in the usual way.  Thus $X\ll Y$ and $X=O(Y)$
both mean that \(|X|\le CY\) for some constant \(C>0\).  Subscripts indicate
the allowed dependence of the implicit constant.
All logarithms are natural. The constant $c_0>0$ below is chosen sufficiently
small and then fixed, and $N$ is sufficiently large in terms of this choice.

\item We write $X\asymp Y$
to mean \(X\ll Y\) and \(Y\ll X\).

\item For every $t\in\R$, we denote by $\lfloor t\rfloor$ the greatest integer not exceeding $t$, and by $\lceil t\rceil$ the least integer no less than $t$.

\item We write \(\Re z\) and \(\arg z\) for the real part
and argument of a complex number \(z\).  Whenever a principal
branch is mentioned, \(\arg z\) is taken in \((-\pi,\pi]\).  In particular,
\(z^{-1/2}\) denotes the inverse square root with respect to the principal
branch.

\end{itemize}

\section{Motivating the van der Corput witness}\label{sec:defining-weight}
The purpose of this section is to ``reverse engineer'' our van der Corput witness. We first recall the construction in~\cite[Section 3]{FanLott}. There we considered
\[
 \frac12\sum_{v,x,y\in\Z}
 U\!\left(\frac{x}{T}\right)U\!\left(\frac{y}{T}\right)
 e^{-v^2/T^2}e\!\left(\theta v^2(x^2+y^2)\right),
\]
where $U$ is a fixed nonzero nonnegative smooth function supported on
$[1,2]$ and $T$ is an auxiliary scale. 
In words, you take the sums of two squares $x^2+y^2$, dilate the set by $v^2$, and then average over $v$. This is what we mean by "averaging over dilations". By taking the real part, collecting like terms, truncating, adding a suitable constant, and normalizing, we obtain a van der Corput witness for the sums of two squares in $[N]$. The goal was therefore reduced to proving a sufficiently strong lower bound for the real part of this sum. To understand what averaging over dilations
achieves, we focus on the complete sum that arises near a rational
frequency $a/q$:
\[
 \frac1{q^3}\sum_{v,x,y\bmod q}e\!\left(\frac{av^2(x^2+y^2)}q\right)
 =\frac1{q^3}\sum_{v\bmod q}G_q(av^2)^2,
\]
where
\[
 G_q(a)=\sum_{x\bmod q}e(ax^2/q)
\]
is a quadratic Gauss sum. The effect of the dilation is already visible when $q$ is an odd prime $p$.
For $(a,p)=1$, we have $G_p(av^2)=p$ when $p\mid v$ and
$G_p(av^2)=G_p(a)$ otherwise, so
\[
 \frac1p+\left(1-\frac1p\right)\frac{G_p(a)^2}{p^2}
 =\begin{cases}
  2/p-1/p^2,&\text{if }p\equiv1\pmod4,\\
  1/p^2,&\text{if }p\equiv3\pmod4.
 \end{cases}
\]
Here we used $G_p(a)^2=p$ or $-p$ when $p\equiv1$ or
$3\pmod4$, respectively. In this way, the dilations divisible by $p$ contribute enough positive
mass to compensate even for the negative case. More generally, this complete sum lies in the sector $|\arg z|<\pi/4$ for every $q$;
see~\cite[Corollary 4.4]{FanLott}.

Now, with the goal of obtaining a van der Corput witness for squares, we remove the $y$-variable. With the same uniform averaging over dilations,
the complete sum becomes
\[
 \frac1{q^2}\sum_{v,x\bmod q}e\!\left(\frac{av^2x^2}q\right)
 =\frac1{q^2}\sum_{v\bmod q}G_q(av^2).
\]
At a prime this gives
\[
 \frac1p+\left(1-\frac1p\right)\frac{G_p(a)}p.
\]
If $p\equiv1\pmod4$ and $a$ is a quadratic nonresidue modulo $p$, then
$G_p(a)=-\sqrt p$, and hence
\[
 \frac1p-\frac{1-1/p}{\sqrt p}<0.
\]
The missing Gauss-sum factor is the problem. Indeed, the negative contribution now
has size $p^{-1/2}$, while only a proportion $p^{-1}$ of the dilations
give the favorable positive value $1$. Simply averaging over more dilations does
not remove this obstruction.

We therefore attach signed weights to the dilations. For a cutoff $V$ to be
chosen below, we seek real coefficients $\lambda_v~(v\le V)$ such that
\[ 
\sum_{v\le V}\lambda_v=1,
\]
and 
\[
 C_q(a)=\frac1q\sum_{v\le V}\lambda_vG_q(av^2)
\]
is close to a positive real number for every relevant $q$, with a uniformly
small error. The normalization fixes $C_1(0)=1$. Write
\begin{equation}\label{eq:divisibility}
 b_d=\sum_{\substack{v\le V\\d\mid v}}\lambda_v
 \qquad(1\le d\le V).
\end{equation}
Then $b_1=1$. M\"obius inversion gives
\begin{equation}\label{eq:lambda}
 \lambda_v
 =\sum_{\substack{d\le V\\v\mid d}}\mu(d/v)b_d
 =\sum_{k\le V/v}\mu(k)b_{vk}
 \qquad(1\le v\le V).
\end{equation}
We may therefore choose $b_d$ with the signed weights $\lambda_v$ determined by Möbius inversion. Notice that $\lambda_v$ has no sign restriction in general, as it is defined in terms of this sum weighted by the M\"obius function. 

We now record some useful properties which the divisibility weights $b_d$ should satisfy. First, taking $0<b_d\le1$ in~\eqref{eq:lambda} gives
the trivial bounds
\begin{equation}\label{eq:cost}
 |\lambda_v|\le V/v,\qquad
 \sum_{v\le V}v|\lambda_v|\le V^2,
\end{equation}
which will help us later when making estimates.

For a general modulus $q$, put
\[
 D=D(q)=\prod_{p^m\parallel q}p^{\lceil m/2\rceil}.
\]
The favorable dilations are precisely those with $D\mid v$, since then
$q\mid v^2$ and $G_q(av^2)/q=1$. Their total weight is $b_D$ if $D\le V$.
Using the identity $G_q(av^2)=G_q(a(v,D)^2)$ from
Lemma~\ref{lem:quadratic-sums} below and the classical M\"obius identity
\[
 \sum_{e\mid n}\mu(e)=
 \begin{cases}
  1,&\text{if }n=1,\\
  0,&\text{otherwise},
 \end{cases}
\]
we obtain
\begin{align}
 C_q(a)
 =\sum_{v\le V}\lambda_v\frac{G_q(av^2)}q&=\sum_{r\mid D}\frac{G_q(ar^2)}q
   \sum_{\substack{v\le V\\(v,D)=r}}\lambda_v\notag\\
 &=\sum_{r\mid D}\frac{G_q(ar^2)}q
   \sum_{\substack{v\le V\\r\mid v}}\lambda_v
   \sum_{e\mid(v/r,D/r)}\mu(e)\notag\\
 &=\sum_{r\mid D}\frac{G_q(ar^2)}q
   \sum_{e\mid D/r}\mu(e)
   \sum_{\substack{v\le V\\re\mid v}}\lambda_v\notag\\
 &=\sum_{r\mid D}\frac{G_q(ar^2)}q
   \sum_{\substack{e\mid D/r\\re\le V}}\mu(e)b_{re}.
\label{eq:sketch-Cq-divisors}
\end{align}
Here the last equality uses~\eqref{eq:divisibility}; the inner
sum over $v$ is empty when $re>V$.
We therefore define $b_d$ for all positive integers $d$ and require sufficient
decay for $d>V$ to extend the inner sum in~\eqref{eq:sketch-Cq-divisors}
to all $e\mid D/r$ with a small absolute error, allowing
us to exploit the structure of the divisor sums and separate out the term $b_D$:
\begin{equation}\label{eq:sketch-Cq-separated}
 C_q(a)
 =b_D+
 \sum_{\substack{r\mid D\\r<D}}
 \frac{G_q(ar^2)}q
 \sum_{e\mid D/r}\mu(e)b_{re}
 +\text{small error}.
\end{equation}
We want the sum of the absolute values of the terms with $r<D$
in~\eqref{eq:sketch-Cq-separated} to be $O(c_0b_D)$. This gives
$C_q(a)=b_D(1+O(c_0))$ up to the truncation error.

Set
\[
 H=\frac{\sqrt N}{\log N},\qquad V=H^{1/8},\qquad
 Y=\sqrt{\frac{\log H}{\log\log H}},
\]
and take $c_0>0$ sufficiently small and fixed. We can now
collect the requirements for $b_d$ before making the definition.
The estimates below are uniform for $1\le q\le H$ and $(a,q)=1$,
with $D=D(q)$.
\begin{enumerate}[label=(\roman*),ref=(\roman*),leftmargin=*]
 \item\label{req:weights-normalization} $b_1=1$ and $0<b_d\le1$.
 These give $\sum_{v\le V}\lambda_v=1$ and the bounds \eqref{eq:cost} for
 $\lambda_v$ above.
 \item\label{req:weights-decay} $b_d$ is small enough for
 $d>V$ so that extending the inner sum in~\eqref{eq:sketch-Cq-divisors}
 to all $e\mid D/r$ costs an absolute error $O(H^{-c/Y})$
 for some fixed $c>0$.
 \item\label{req:weights-arithmetic} We have
 \[
  \sum_{\substack{r\mid D\\r<D}}
  \left|\frac{G_q(ar^2)}q
  \sum_{e\mid D/r}\mu(e)b_{re}\right|
  \ll c_0b_D.
 \]
\end{enumerate}

For simplicity, first take $b_d=\beta_d$ to be completely multiplicative,
with $\beta_1=1$ and $0<\beta_p\le1$. The following lemma gives
the divisor-sum estimate needed to test~\ref{req:weights-arithmetic}.

\begin{lemma}
\label{lem:multiplicative-divisor-bound}
Let $\beta_d$ be completely multiplicative, with $\beta_1=1$ and
$0<\beta_p\le1$ for every prime $p$. If $D$ is a positive integer
such that $\sqrt p\,\beta_p>1$ for every $p\mid D$, then
\begin{equation}\label{eq:sketch-product-bound}
\begin{aligned}
 \sum_{\substack{r\mid D\\r<D}}\sqrt{\frac rD}
 \left|\sum_{e\mid D/r}\mu(e)\beta_{re}\right|
 &\le\beta_D\left(
 \prod_{p\mid D}\left(1+\frac{1-\beta_p}{\sqrt p\,\beta_p-1}\right)-1
 \right).
\end{aligned}
\end{equation}
\end{lemma}

\begin{proof}
For each $r\mid D$, complete multiplicativity of $\beta_d$ gives
\[
 \sum_{e\mid D/r}\mu(e)\beta_{re}
 =\beta_r\sum_{e\mid D/r}\mu(e)\beta_e
 =\beta_r\prod_{p\mid D/r}\bigl(1-\beta_p\bigr).
\]
Writing $h=D/r$ and using $\beta_D=\beta_r\beta_h$, we obtain
\[
 \frac1{\beta_D}\sum_{\substack{r\mid D\\r<D}}
 \sqrt{\frac rD}
 \left|\sum_{e\mid D/r}\mu(e)\beta_{re}\right|
 =\sum_{\substack{h\mid D\\h>1}}
 \frac{1}{\sqrt h\,\beta_h}\prod_{p\mid h}(1-\beta_p).
\]
We factor the divisor sum over the primes dividing $D$. For each
prime $p\mid D$, its positive powers contribute at most
\[
 (1-\beta_p)\sum_{\substack{j\ge1\\p^j\mid D}}
 \frac1{(\sqrt p\,\beta_p)^j}
 \le(1-\beta_p)\sum_{j\ge1}\frac1{(\sqrt p\,\beta_p)^j}
 =\frac{1-\beta_p}{\sqrt p\,\beta_p-1}.
\]
Multiplying the resulting factors accounts for divisors involving several
primes; subtracting $1$ removes $h=1$, which corresponds to $r=D$.
\end{proof}

For $r\mid D$, the Gauss-sum estimate gives
\[
 \frac{|G_q(ar^2)|}{q}\le\sqrt{\frac{2r}{D}}.
\]
Thus Lemma~\ref{lem:multiplicative-divisor-bound} bounds the sum
in~\ref{req:weights-arithmetic} by $\sqrt2$ times the right-hand side
of~\eqref{eq:sketch-product-bound}. To obtain~\ref{req:weights-arithmetic},
we want the $\beta_p$ close enough to $1$ that the product on the
right-hand side is $1+O(c_0)$. This makes the sum of the absolute values
of the terms with $r<D$ in~\eqref{eq:sketch-Cq-separated} at most
$O(c_0\beta_D)$, isolating the contribution from $\beta_D$.

However, requiring this product to be $1+O(c_0)$ prevents the completely
multiplicative weights $\beta_D$ from having the pointwise decay beyond $V$
that we plan to use for~\ref{req:weights-decay}. For example, take a squarefree $D$ with $V<D\le H$ built
from primes $p\ll\log H$. Such a $D$ can be obtained by multiplying
consecutive primes until the product first exceeds $V$, using Chebyshev's
estimates. To make the product in~\eqref{eq:sketch-product-bound}
equal to $1+O(c_0)$, we need
\[
 \sum_{p\mid D}\frac{1-\beta_p}{\sqrt p\,\beta_p-1}\ll c_0.
\]
But then
\begin{align*}
 \log\frac1{\beta_D}
 &=\sum_{p\mid D}\log\frac1{\beta_p}
 \le\sum_{p\mid D}\frac{1-\beta_p}{\beta_p}\le\max_{p\mid D}\sqrt p\,
     \sum_{p\mid D}\frac{1-\beta_p}{\sqrt p\,\beta_p-1}
 \ll c_0\sqrt{\log H}.
\end{align*}
Thus a single multiplicative choice cannot make the product in~\eqref{eq:sketch-product-bound} equal to $1+O(c_0)$
while simultaneously ensuring the decay of size
\[
 H^{-c/Y}=\exp\bigl(-c\sqrt{\log H\log\log H}\bigr)
\]
that we seek beyond $V$, for a fixed $c>0$.

We therefore take a positive average of completely multiplicative
weights $\beta_d(u)$. Write $d=d_1d_2$, where $d_1$ is $Y^2$-smooth and $d_2$ is $Y^2$-rough, and define
\begin{equation}\label{eq:sketch-beta-family}
 \beta_d(u)=d_1^{-c_0/Y}d_2^{-u/(2\log Y)},\qquad 0\le u\le1,
\end{equation}
and
\begin{equation}\label{eq:b}
 b_d=c_0Y\int_0^1\beta_d(u)u^{c_0Y-1}\,du.
\end{equation}

There is a natural probabilistic interpretation of the weights $\beta_d(u)$. 
In the standard weighted probability model for smooth numbers, a randomly chosen $y$-smooth number $\mathbf{n}$ takes value $n\in\N$ with probability $n^{-\alpha}/\zeta(\alpha,y)$ if $n$ is $y$-smooth and 0 otherwise, where $\alpha>0$ and
\[
\zeta(\alpha,y):=\sum_{n\ge1}\frac{\mathbf1_{n\text{ is $y$-smooth}}}{n^\alpha}=\prod_{p\le y}\big(1-p^{-\alpha}\big)^{-1}.
\]
Thus 
\[
\mathbb P(d\mid \mathbf{n})=\sum_{m\ge1}\mathbb P(\mathbf{n}=md)\mathbf1_{m\text{ is $y$-smooth}}=d^{-\alpha}
\]
for every $y$-smooth number $d$. 

In our case, we extend this model to integers whose prime factors lie in a finite set $\mathcal P$ containing all primes $p\le y=Y^2$. Fix $u\in(0,1]$, and choose $\mathbf{n}$ so that its prime exponents are independent and satisfy
\[
\mathbb P(p^j\mid \mathbf{n})=\begin{cases}
    p^{-j\alpha},&\text{if $p\le y$},\\
    p^{-ju/\log y},&\text{if $p>y$},
\end{cases}
\]
for $p\in\mathcal P$ and $j\in\N$, where $\alpha=c_0/Y$ with $c_0>0$ sufficiently small. Then $\mathbb P(d\mid \mathbf{n})=\beta_d(u)$ whenever all prime factors of $d$ lie in $\mathcal P$. We may choose $\mathcal P$ to contain the prime factors of any finite collection of integers under consideration. Note that for $p\le y$,
\[
\mathbb P(p\nmid \mathbf{n})=1-p^{-\alpha}\sim\frac{c_0\log p}{Y}.
\]
Thus, each prime $p\le y$ divides $\mathbf{n}$ with probability close to $1$. This choice of $\alpha$ gives decay in the smooth part $d_1$ while ensuring, as we shall see later, that
\[
\sum_{\substack{p\mid D\\p\le Y^2}}
\frac{1-\beta_p(u)}{\sqrt p\,\beta_p(u)-1}\ll c_0.
\]

Now let $\mathbf U\sim \text{Beta}(\kappa,1)$ with probability density function $f_{\mathbf U}(u)=\kappa u^{\kappa-1}$ for $u\in[0,1]$, where $\kappa=c_0Y$. After sampling $\mathbf U$, choose $\mathbf{n}$ by the above model with $u=\mathbf U$. Then $\mathbb P(d\text{ divides }\mathbf{n}\mid \mathbf U=u)=\beta_d(u)$ whenever all prime factors of $d$ lie in $\mathcal P$. For these $d$, it follows that
\[
\mathbb P(d\mid \mathbf{n})=\mathbb E[\beta_d(\mathbf U)]=\int_{0}^{1}\beta_d(u)f_{\mathbf U}(u)\,du=b_d.
\]
When the rough part $d_2$ is large, $\beta_d(u)$ decays rapidly as $u$ increases, while $f_{\mathbf U}(u)$ has most mass near $u=1$. Together with the decay in the smooth part $d_1$, this ensures that $b_d$ is tiny for $d>V$, as required for~\ref{req:weights-decay}. In addition, replacing $\beta_d(u)$ by $\beta_d((1-\epsilon)u)$ changes the integral defining $b_d$ by at most a factor $1+O(c_0)$:
\begin{align*}
b_d\le c_0Y\int_0^1\beta_d((1-\epsilon)u)u^{c_0Y-1}\,du=\frac{c_0Y}{(1-\epsilon)^{c_0Y}}\int_0^{1-\epsilon}\beta_d(w)w^{c_0Y-1}\,dw&\le (1-\epsilon)^{-c_0Y}b_d\\
&=(1+O(c_0))b_d,    
\end{align*}
where $0<\epsilon\ll 1/Y$. This allows us to show that
\[
\sum_{\substack{p\mid D\\p>Y^2}}
\frac{1-\beta_p(u)}{\sqrt p\,\beta_p(u)-1}\ll c_0
\]
on average over $u\in[0,1]$ with respect to the measure 
\[
\frac{\beta_D(u)f_{\mathbf U}(u)}{b_D}\,du.
\] 
The same change of variables will be used in the proof of Lemma \ref{lem:averaged-weights} below.

We do not require each multiplicative weight $\beta_d(u)$ to
satisfy~\ref{req:weights-arithmetic}; we will obtain this bound only
after averaging on $u\in[0,1]$.
The split between small and large primes in~\eqref{eq:sketch-beta-family}
helps us balance~\ref{req:weights-decay} and~\ref{req:weights-arithmetic}.
For $p\le Y^2$, the weights $\beta_p(u)=p^{-c_0/Y}$ stay close to $1$,
keeping their part of the product in~\eqref{eq:sketch-product-bound} close
to $1$. For $p>Y^2$, the weights $\beta_p(u)=p^{-u/(2\log Y)}$ can decay
more rapidly as $u$ increases, helping with~\ref{req:weights-decay}.
Each $\beta_d(u)$ is completely multiplicative, so we can apply
Lemma~\ref{lem:multiplicative-divisor-bound} before averaging.

We record the resulting properties in a form that can be used later
in the proof of Proposition~\ref{prop:arithmetic}.

\begin{lemma}[Properties of the averaged weights]
\label{lem:averaged-weights}
Let the parameters and weights be as above. For sufficiently large $H$
depending on the fixed constant $c_0>0$, we have
\[
 b_1=1,\qquad 0<b_d\le1\qquad(d\ge1),
\]
and
\begin{equation}\label{eq:weights-tail}
 b_d\le2H^{-c_0/(8Y)}\qquad(d>V).
\end{equation}
Moreover, for every positive integer $D$,
\begin{equation}\label{eq:weights-divisor-bound}
 \sum_{\substack{r\mid D\\r<D}}\sqrt{\frac rD}
 \left|\sum_{e\mid D/r}\mu(e)b_{re}\right|
 \ll c_0b_D,
\end{equation}
where the implied constant is absolute.
\end{lemma}

\begin{proof}
We check the properties in the same order. Each $\beta_d(u)$ lies
in $(0,1]$, and $\beta_1(u)=1$. Since
\[
 c_0Y\int_0^1u^{c_0Y-1}\,du=1,
\]
the average has $b_1=1$ and $0<b_d\le1$.

For each $d\in\N$, write $d=d_1d_2$ with $d_1$ and $d_2$ being the $Y^2$-smooth and $Y^2$-rough parts of $d$, respectively. For sufficiently large $H$, if
$u\ge Y^{-1/2}$, then $u/(2\log Y)\ge c_0/Y$ and therefore
\[
 \beta_d(u)=d_1^{-c_0/Y}d_2^{-u/(2\log Y)}\le d^{-c_0/Y}.
\]
By definition \eqref{eq:b},
\begin{equation}
 b_d\le c_0Y\int_0^{Y^{-1/2}}u^{c_0Y-1}\,du
       +d^{-c_0/Y}c_0Y\int_{Y^{-1/2}}^1u^{c_0Y-1}\,du
 \le Y^{-c_0Y/2}+d^{-c_0/Y}.
\end{equation}
Since $Y^2=(\log H)/(\log\log H)$ and
$\log Y\sim\frac12\log\log H$, we have
\[
 \frac Y2\log Y\ge\frac{\log H}{8Y}
\]
for sufficiently large $H$. Thus $b_d\le2H^{-c_0/(8Y)}$ whenever $d>H^{1/8}$.
This proves~\eqref{eq:weights-tail}.

For sufficiently large $H$, both $c_0/Y$ and $1/(2\log Y)$ are at most
$1/4$, so $\beta_p(u)\ge p^{-1/4}$ and $\sqrt p\,\beta_p(u)\ge p^{1/4}>1$.
Applying Lemma~\ref{lem:multiplicative-divisor-bound} and using
$1+t\le e^t$, we obtain
\[
 \sum_{\substack{r\mid D\\r<D}}\sqrt{\frac rD}
 \left|\sum_{e\mid D/r}\mu(e)\beta_{re}(u)\right|
 \le\beta_D(u)\left(
 \exp\left(\sum_{p\mid D}\frac{1-\beta_p(u)}{\sqrt p\,\beta_p(u)-1}\right)-1
 \right).
\]
By definition~\eqref{eq:b} and the triangle inequality, this gives
\begin{equation}\label{eq:averaged-divisor-error}
\begin{split}
 &\sum_{\substack{r\mid D\\r<D}}\sqrt{\frac rD}
 \left|\sum_{e\mid D/r}\mu(e)b_{re}\right|\le c_0Y\int_0^1\beta_D(u)\left(
 \exp\left(\sum_{p\mid D}\frac{1-\beta_p(u)}{\sqrt p\,\beta_p(u)-1}\right)-1
 \right)u^{c_0Y-1}\,du.
\end{split}
\end{equation}

We split the prime factors $p\mid D$ according to their sizes relative to $Y^2$. For $0\le s\le1/4$, the inequalities
\[
 1-p^{-s}\le s\log p,\qquad
 p^{1/2-s}-1\ge(1-2^{-1/4})p^{1/2-s}
\]
give
\[
 \frac{1-p^{-s}}{p^{1/2-s}-1}
 \ll s(\log p)p^{s-1/2},
\]
with an absolute implicit constant. If $p\le Y^2$, then $\beta_p(u)=p^{-s}$ with $s=c_0/Y$. Since
$Y^{2c_0/Y}=\exp(2c_0\log Y/Y)=O(1)$, we get
\[
 \frac{1-\beta_p(u)}{\sqrt p\,\beta_p(u)-1}
 \ll\frac{c_0\log p}{Y\sqrt p}.
\]
Summing this over $p\le Y^2$ using Chebyshev's estimate and partial summation, we obtain
\[
\sum_{\substack{p\mid D\\p\le Y^2}}
\frac{1-\beta_p(u)}{\sqrt p\,\beta_p(u)-1}\ll\frac{c_0}{Y}\sum_{p\le Y^2}\frac{\log p}{\sqrt p}\ll c_0.
\]

Now consider the remaining primes, namely those with $p>Y^2$.
We apply the same inequalities but with $s=u/(2\log Y)$.
The exponent $u/(2\log Y)-1/2$ is negative, so
\[
 p^{u/(2\log Y)-1/2}
 \le\big(Y^2\big)^{u/(2\log Y)-1/2}
 =\frac{e^u}{Y}\le\frac eY.
\]
As a result,
\[
 \frac{1-\beta_p(u)}{\sqrt p\,\beta_p(u)-1}
 \ll\frac{u\log p}{Y\log Y}.
\]
Write $D=D_1D_2$ with $D_1$ and $D_2$ being the $Y^2$-smooth and $Y^2$-rough parts of $D$, respectively. Then
\[
\sum_{\substack{p\mid D\\p>Y^2}}\frac{1-\beta_p(u)}{\sqrt p\,\beta_p(u)-1}\ll \frac{u}{Y\log Y}\sum_{\substack{p\mid D\\p>Y^2}}\log p\le\frac{u\log D_2}{Y\log Y}.
\]

Gathering the estimates above, we have
\[
 \sum_{p\mid D}\frac{1-\beta_p(u)}{\sqrt p\,\beta_p(u)-1}
 \le C\left(c_0+\frac{u\log D_2}{2Y\log Y}\right)
\]
for some absolute constant $C>0$. Since
\[
 \beta_D(u)=D_1^{-c_0/Y}D_2^{-u/(2\log Y)},
\]
it follows that
\[
 \beta_D(u)\left(
 \exp\left(\sum_{p\mid D}\frac{1-\beta_p(u)}{\sqrt p\,\beta_p(u)-1}\right)-1
 \right)
 \le e^{Cc_0}\beta_D\bigl((1-C/Y)u\bigr)-\beta_D(u).
\]
Inserting this into~\eqref{eq:averaged-divisor-error} gives
\[
 \begin{split}
 &\sum_{\substack{r\mid D\\r<D}}\sqrt{\frac rD}
 \left|\sum_{e\mid D/r}\mu(e)b_{re}\right|\le e^{Cc_0}c_0Y\int_0^1
       \beta_D\bigl((1-C/Y)u\bigr)u^{c_0Y-1}\,du-b_D.
 \end{split}
\]
For sufficiently large $H$, we have $0<C/Y<1/2$. Changing variables by $w=(1-C/Y)u$, we have
 \begin{align*}
 c_0Y\int_0^1\beta_D\bigl((1-C/Y)u\bigr)u^{c_0Y-1}\,du
 &=\frac{c_0Y}{(1-C/Y)^{c_0Y}}
   \int_0^{1-C/Y}\beta_D(w)w^{c_0Y-1}\,dw\\
   &\le\frac{b_D}{(1-C/Y)^{c_0Y}}.
 \end{align*}
Therefore,
\[
 \sum_{\substack{r\mid D\\r<D}}\sqrt{\frac rD}
 \left|\sum_{e\mid D/r}\mu(e)b_{re}\right|
 \le b_D\left(e^{Cc_0}(1-C/Y)^{-c_0Y}-1\right)
 \ll c_0b_D,
\]
since
\[
 e^{Cc_0}(1-C/Y)^{-c_0Y}-1
 =(1+O(Cc_0))\left(1+O\left(c_0Y\cdot\frac{C}{Y}\right)\right)-1
 \ll c_0.
\]
This proves~\eqref{eq:weights-divisor-bound}.
\end{proof}

The normalization in Lemma~\ref{lem:averaged-weights}
gives~\ref{req:weights-normalization}. For~\ref{req:weights-decay},
multiplying the bound~\eqref{eq:weights-tail} by $|G_q(ar^2)|/q$
and summing over $r\mid D$ and $e\mid D/r$ with $re>V$
increases it by at most a factor $\exp(O(\sqrt{\log H}))$,
which is absorbed by the decay. Lemma~\ref{lem:divisor-completion}
below supplies the details. Thus extending the inner sum
in~\eqref{eq:sketch-Cq-divisors} to all $e\mid D/r$
costs $O(H^{-c/Y})$ for some fixed $c>0$, including when $D>V$.
Finally,~\eqref{eq:weights-divisor-bound} and the Gauss-sum estimate
$|G_q(ar^2)|/q\le\sqrt{2r/D}$ give~\ref{req:weights-arithmetic}.

These estimates give the following proposition, whose proof is completed
in Section~\ref{sec:complete-sum}.

\begin{proposition}\label{prop:arithmetic}
Fix a sufficiently small absolute constant $c_0>0$, and let the parameters
and weights be as above. Put $c=c_0/64$. Then for all sufficiently large $H$ depending on $c_0$, and
uniformly for $1\le q\le H$ and $a\in\Z$ with $(a,q)=1$, we have
\begin{equation}\label{eq:arithmetic-approx}
 C_q(a)
 =b_D\bigl(1+O(c_0)\bigr)
 +O\big(H^{-c/Y}\big),
\end{equation}
where $D=D(q)$ and the implied constants are absolute.
\end{proposition}

Thus $C_q(a)$ has the positive main term $b_D$ we sought, with a small
relative error and a uniformly small absolute error.

Now we are ready to define the van der Corput witness for the squares. Set 
\begin{equation}\label{eq:gaussian-sum}
 \mathcal B_H(\theta)
 =\sum_{v\le H^{1/8}}v\lambda_v
   \sum_{\substack{x\in\Z\\0<v^2x^2\le N}}
   e^{-v^2x^2/H^2}e(\theta v^2x^2).
\end{equation}

The factor $v$ cancels the factor $1/v$ from the change of variables
$t=vx$ in the corresponding integral, so that the complete sum is
exactly $C_q(a)$. For $\eta\in\R$, write
\begin{equation}\label{eq:gaussian-integral}
 I_H(\eta)=\int_{\R}e^{-x^2/H^2}e(\eta x^2)\,dx
 =\frac{\sqrt\pi H}{\sqrt{1-2\pi i\eta H^2}},
\end{equation}
where the square root is chosen to have positive real part.
The following proposition, proved in Section~\ref{sec:poisson},
relates $\mathcal B_H$ to $C_q(a)$ and $I_H(\eta)$.

\begin{proposition}\label{prop:gaussian}
Let $1\le q\le H$, $(a,q)=1$, and $|\eta|\le2/(qH)$. Then
\begin{equation}\label{eq:global-gaussian}
 \mathcal B_H(a/q+\eta)=C_q(a)I_H(\eta)+O(H^{3/4}),
\end{equation}
where the implied constant is absolute.
\end{proposition}

The argument of $I_H(\eta)$ lies in $(-\pi/4,\pi/4)$. For sufficiently small $c_0$,
the same is true of $b_D(1+O(c_0))$, so their product has positive real
part. The errors are $O(H^{1-c/Y})$, and Dirichlet approximation
covers every $\theta\in\T$. Together with $C_1(0)=b_1=1$, this gives
the following proposition, whose proof is given in
Section~\ref{sec:finishing}.

\begin{proposition}\label{prop:uniform}
For a sufficiently small choice of $c_0$, we have
\begin{equation}\label{eq:normalization}
 \mathcal B_H(0)=\sqrt\pi H+O(H^{3/4}).
\end{equation}
In particular, $\mathcal B_H(0)>0$ and $\mathcal B_H(0)\asymp H$
for sufficiently large $H$. Moreover, there are constants
$c,C>0$ such that
\begin{equation}
 \Re\mathcal B_H(\theta)\ge-CH^{1-c/Y}
 \qquad(\theta\in\T).
\end{equation}
\end{proposition}

\begin{proof}[Proof of Theorem~\ref{thm:main} assuming
Proposition~\ref{prop:uniform}]
Let $c,C>0$ be as in Proposition~\ref{prop:uniform}, and put
\[
 P(\theta)=
 \frac{CH^{1-c/Y}+\Re\mathcal B_H(\theta)}
      {CH^{1-c/Y}+\mathcal B_H(0)}.
\]
By Proposition~\ref{prop:uniform}, we have $P\ge0$ and $P(0)=1$.
Pairing $x$ and $-x$ in~\eqref{eq:gaussian-sum} and grouping
the terms with $v|x|=n$, we obtain
\[
 P(\theta)=a_0+
 \sum_{1\le n\le\sqrt N}a_n\cos(2\pi n^2\theta),
\]
where
\[
 a_0=\frac{CH^{1-c/Y}}{CH^{1-c/Y}+\mathcal B_H(0)}
\]
and
\[
 a_n=
 \frac{2e^{-n^2/H^2}}{CH^{1-c/Y}+\mathcal B_H(0)}
 \sum_{\substack{v\mid n\\v\le H^{1/8}}}v\lambda_v
 \qquad(1\le n\le\sqrt N).
\]
Since $\mathcal B_H(0)\asymp H$, we have
\[
 0<a_0\ll H^{-c/Y}.
\]
Finally, $\log H=\frac12\log N-\log\log N$, so
\[
 \frac{\log H}{Y}
 =\sqrt{\log H\log\log H}
 \asymp\sqrt{\log N\log\log N}.
\]
This proves the theorem after decreasing the constant
in the exponent.
\end{proof}

\section{Finishing the complete sum estimate}\label{sec:complete-sum}
We prove Proposition~\ref{prop:arithmetic}. We first establish the
Gauss-sum identities and bound used in Section~\ref{sec:defining-weight},
and then estimate the error from extending the divisor sum
in~\eqref{eq:sketch-Cq-divisors}.

\begin{lemma}\label{lem:quadratic-sums}
Let $q\in\N$ and $a\in\Z$ with $(a,q)=1$. Let $D=D(q)$ be as above. Then
\[
 G_q(av^2)=G_q\bigl(a(v,D)^2\bigr)\qquad(v\ge1),
 \qquad G_q(aD^2)=q.
\]
In addition, we have
\begin{equation}\label{eq:gauss-bound}
 \frac{|G_q(ar^2)|}{q}\le\sqrt{\frac{2r}{D}}
\end{equation}
for every $r\mid D$.
\end{lemma}

\begin{proof}
The identity $G_q(aD^2)=q$ is trivial as $q\mid D^2$. For each $v\ge1$, put
\begin{equation}\label{eq:reduced-modulus}
 q_v=\frac{q}{(v^2,q)},\qquad
 a_v=\frac{av^2}{(v^2,q)}.
\end{equation}
Then $(a_v,q_v)=1$ and 
\begin{equation}\label{eq:reduced-sum}
 \frac{G_q(av^2)}q=\frac{G_{q_v}(a_v)}{q_v}.
\end{equation}
This reduction will also be used in Section~\ref{sec:poisson}.

Put $g=(v,D)$. Since $q\mid D^2$, we have 
\[
(g^2,q)=((v,D)^2,q)=(v^2,(D^2,q))=(v^2,q),
\]
so that $q_g=q_v$. Moreover, setting $h=(v/g,q_v)$, we have $h^2g^2\mid v^2$ and $h(v^2,q)\mid q$. Thus $h(g^2,(v^2,q))\mid (v^2,q)$. Since $(v^2,q)\mid g^2$, we have $h(v^2,q)\mid (v^2,q)$, which implies $h=1$. So $(v/g,q_v)=1$. Consequently, since $a_v=a_g(v/g)^2$, changing variables by $x\mapsto(v/g)x$ in~\eqref{eq:reduced-sum} gives 
\[
G_q(av^2)=q\cdot\frac{G_{q_v}(a_v)}{q_v}=q\cdot\frac{G_{q_v}(a_g)}{q_v}=q\cdot\frac{G_{q_g}(a_g)}{q_g}=G_q(ag^2).
\]

Finally, by the square-root cancellation bound for quadratic Gauss sums:
\begin{equation}\label{eq:primitive-gauss}
 \left|\frac1q\sum_{x\bmod q}e\bigl((ax^2+bx)/q\bigr)\right|
 \le\sqrt{\frac2q},
\end{equation}
where $q\in\N$, $a\in\Z$ with $(a,q)=1$, and $b\in\Z$, we have, with $v=r\mid D$, 
\[
\frac{|G_q(ar^2)|}{q}=\frac{|G_{q_r}(a_r)|}{q_r}\le \sqrt{\frac{2}{q_r}}\le \sqrt{\frac{2r}{D}},
\]
since $q_r\ge D(q_r)=D/g=D/r$.
\end{proof}

The next lemma bounds the error from removing the restriction $re\le V$
in~\eqref{eq:sketch-Cq-divisors}.
\begin{lemma}\label{lem:divisor-completion}
For all sufficiently large $H$ depending on $c_0$, and
uniformly for $1\le q\le H$ and $a\in\Z$ with $(a,q)=1$, we have
\begin{equation}\label{eq:divisor-average}
 C_q(a)=\sum_{r\mid D}\frac{G_q(ar^2)}q
          \sum_{e\mid D/r}\mu(e)b_{re}
       +O\big(H^{-c/Y}\big),
\end{equation}
where $c=c_0/64$.
\end{lemma}

\begin{proof}
For $q=1$, both $C_1(a)$ and the main term in~\eqref{eq:divisor-average} equal $b_1=1$, so we may assume $q\ge2$.
Starting from~\eqref{eq:sketch-Cq-divisors}, we now remove the restriction
$re\le H^{1/8}$. For $re>H^{1/8}$, we have
$b_{re}\le2H^{-c_0/(8Y)}$ by~\eqref{eq:weights-tail}.
Using~\eqref{eq:gauss-bound} and then writing $h=D/r$, we see that the tail of
\eqref{eq:sketch-Cq-divisors} over $re>H^{1/8}$ is at most
\begin{align*}
2H^{-c_0/(8Y)}\sum_{r\mid D}\frac{|G_q(ar^2)|}{q}\sum_{e\mid D/r}|\mu(e)|
 &\le2\sqrt2 H^{-c_0/(8Y)}\sum_{h\mid D}\frac1{\sqrt h}
       \sum_{e\mid h}|\mu(e)|\notag\\
 &=2\sqrt2 H^{-c_0/(8Y)}\prod_{p^k\parallel D}
     \left(1+2\sum_{j=1}^kp^{-j/2}\right)\notag\\
 &\le2\sqrt2 H^{-c_0/(8Y)}\prod_{p\mid q}
     \left(1+\frac2{\sqrt p-1}\right)\notag\\
     &\ll H^{-c_0/(8Y)}\exp\left(2\sum_{p\le p_{\omega(q)}}\frac{1}{\sqrt{p}-1}\right)\notag\\
     &= H^{-c_0/(8Y)}\exp\left(O\left(\sqrt{\omega(q)\log(3\omega(q))}\right)\right)\notag\\
     &=H^{-c_0/(8Y)}\exp\left(O\left(\sqrt{\log(3q)}\right)\right)\notag\\
 &\le H^{-c_0/(8Y)}\exp\bigl(O(\sqrt{\log H})\bigr)
 \le H^{-c/Y},
\end{align*}
where $\omega(q)$ denotes the number of distinct prime factors of $q$ and satisfies $\omega(q)\ll \log(3q)/\log\log(3q)$, and $p_{\omega(q)}$ is the $\omega(q)$-th prime and satisfies $p_{\omega(q)}\ll \omega(q)\log(3\omega(q))$. Since
\[
 \frac{c_0\log H}{8Y}=\frac{c_0}{8}\sqrt{\log H\log\log H},
\]
the factor $\exp(O(\sqrt{\log H}))$ is absorbed with $c=c_0/64$
for sufficiently large $H$. This proves~\eqref{eq:divisor-average}.
\end{proof}

\begin{proof}[Proof of Proposition~\ref{prop:arithmetic}]
The term with $r=D$ in~\eqref{eq:divisor-average} is exactly $b_D$, since
$G_q(aD^2)=q$. Applying~\eqref{eq:gauss-bound} to the remaining terms and
then invoking~\eqref{eq:weights-divisor-bound}, we obtain
\begin{align*}
|C_q(a)-b_D|
 &\le\sum_{\substack{r\mid D\\r<D}}\frac{|G_q(ar^2)|}{q}
     \left|\sum_{e\mid D/r}\mu(e)b_{re}\right|
     +O\big(H^{-c/Y}\big)\\
 &\le\sqrt2\sum_{\substack{r\mid D\\r<D}}\sqrt{\frac rD}
     \left|\sum_{e\mid D/r}\mu(e)b_{re}\right|
     +O\big(H^{-c/Y}\big)\\
 &\ll c_0b_D+H^{-c/Y}.
\end{align*}
This proves~\eqref{eq:arithmetic-approx}.
\end{proof}

\section{Applying the Poisson summation formula}\label{sec:poisson}
We prove Proposition~\ref{prop:gaussian} by applying Poisson summation
to $\mathcal B_H$, defined in~\eqref{eq:gaussian-sum}.

\begin{proof}[Proof of Proposition~\ref{prop:gaussian}]
Put $\theta=a/q+\eta$. We use the bounds~\eqref{eq:cost} with $V=H^{1/8}$.
We first restore $x=0$ and remove the cutoff $v^2x^2\leq N$. Restoring $x=0$ changes the sum by
\[\sum_{v\leq H^{1/8}}v\lambda_v=O(H^{1/4}).\]
The contribution from the tail given by $v^2x^2>N$ is also negligible. Indeed, if $v^2x^2>N$, then
\[
 e^{-v^2x^2/H^2}
 \le e^{-N/(2H^2)}e^{-v^2x^2/(2H^2)}.
\]
Since $v\le H^{1/8}$, an integral comparison gives
$\sum_{x\in\Z}e^{-v^2x^2/(2H^2)}\ll H/v$. Hence the tail contribution
is at most
\[
 \begin{aligned}
 e^{-N/(2H^2)}\sum_{v\le H^{1/8}}v|\lambda_v|
       \sum_{x\in\Z}e^{-v^2x^2/(2H^2)}\ll H^{5/4}e^{-(\log N)^2/2}\ll1,
 \end{aligned}
\]
where we used $N/H^2=(\log N)^2$ and~\eqref{eq:cost}.
Thus
\[
 \mathcal B_H(\theta)
 =\sum_{v\le H^{1/8}}v\lambda_v
   \sum_{x\in\Z}e^{-v^2x^2/H^2}e(\theta v^2x^2)
 +O(H^{1/4}).
\]

Fix $v\le H^{1/8}$ and recall from~\eqref{eq:reduced-modulus} that
\[
 q_v=\frac q{(v^2,q)},\qquad
 a_v=\frac{av^2}{(v^2,q)},\qquad
 \frac{av^2}{q}=\frac{a_v}{q_v},\qquad (a_v,q_v)=1.
\]
Define
\[
 \Phi_v(x)=e^{-v^2x^2/H^2}e(\eta v^2x^2).
\]
This belongs to $\Sch(\R)$, since each derivative is a polynomial
times $\Phi_v$ and $|\Phi_v(x)|=e^{-v^2x^2/H^2}$.
On the residue class $x\equiv r\pmod*{q_v}$, the rational phase
$e(a_vx^2/q_v)$ equals $e(a_vr^2/q_v)$. Splitting into these
classes and applying~\eqref{eq:residue-poisson}, we obtain
\begin{equation}\label{eq:gaussian-poisson}
 \begin{aligned}
 \sum_{x\in\Z}e(a_vx^2/q_v)\Phi_v(x)
 &=\sum_{r\bmod q_v}e(a_vr^2/q_v)
      \sum_{\substack{x\in\Z\\x\equiv r\pmod*{q_v}}}\Phi_v(x)\\
 &=\frac1{q_v}\sum_{\ell\in\Z}
   \Bigg(\sum_{r\bmod q_v}
     e\!\left(\frac{a_vr^2+\ell r}{q_v}\right)\Bigg)
   \widehat\Phi_v\!\left(\frac\ell{q_v}\right).
 \end{aligned}
\end{equation}

The zero frequency gives the main term. Indeed,
\[
 \frac1{q_v}\sum_{r\bmod q_v}e(a_vr^2/q_v)
 =\frac{G_{q_v}(a_v)}{q_v}
 =\frac{G_q(av^2)}q
\]
by~\eqref{eq:reduced-sum}. Also, changing variables by $y=vx$ gives
\[
 v\widehat\Phi_v(0)
 =v\int_{\R}e^{-v^2x^2/H^2}e(\eta v^2x^2)\,dx
 =\int_{\R}e^{-y^2/H^2}e(\eta y^2)\,dy.
\]
Thus, after multiplication by $v$, the contribution from $\ell=0$
in~\eqref{eq:gaussian-poisson} is
\[
 \frac{G_q(av^2)}q
 \int_{\R}e^{-x^2/H^2}e(\eta x^2)\,dx.
\]

It remains to bound the contribution from $\ell\ne0$.
By the formula for the Fourier transform of the Gaussian, we have
\begin{equation}
 |\widehat\Phi_v(\xi)|
 =\frac{\sqrt\pi H}{v(1+4\pi^2\eta^2H^4)^{1/4}}
 \exp\!\left(
   -\frac{\pi^2H^2\xi^2}{v^2(1+4\pi^2\eta^2H^4)}
 \right).
\end{equation}
Since $(a_v,q_v)=1$, the Gauss-sum bound~\eqref{eq:primitive-gauss}
gives, for every $\ell\in\Z$,
\[
 \left|\frac1{q_v}\sum_{r\bmod q_v}
 e\!\left(\frac{a_vr^2+\ell r}{q_v}\right)\right|
 \le\sqrt{\frac2{q_v}}.
\]
Therefore the total contribution from $\ell\ne0$, after
multiplication by $v$, is at most
\[
 \frac{\sqrt{2\pi}\,H}
      {\sqrt{q_v}(1+4\pi^2\eta^2H^4)^{1/4}}
 \sum_{\ell\ne0}
 \exp\!\left(
  -\frac{\pi^2H^2\ell^2}
        {v^2q_v^2(1+4\pi^2\eta^2H^4)}
 \right).
\]
For every $A>0$, monotonicity of the Gaussian gives
\[
 \sum_{\ell\ne0}e^{-A^2\ell^2}
 =2\sum_{\ell\ge1}e^{-A^2\ell^2}
 \le2\int_0^\infty e^{-A^2x^2}\,dx
 =\frac{\sqrt\pi}{A}.
\]
Applying this with
$A=\pi H/(vq_v\sqrt{1+4\pi^2\eta^2H^4})$ bounds the nonzero
frequencies by
\[
 \ll v\sqrt{q_v}(1+4\pi^2\eta^2H^4)^{1/4}.
\]
Combining the zero and nonzero frequencies, we have proved
\[
 \begin{aligned}
 v\sum_{x\in\Z}e(a_vx^2/q_v)\Phi_v(x)
 &=\frac{G_q(av^2)}q
   \int_{\R}e^{-x^2/H^2}e(\eta x^2)\,dx+O\big(v\sqrt{q_v}(1+4\pi^2\eta^2H^4)^{1/4}\big).
 \end{aligned}
\]

We now sum against $\lambda_v$. The integral is independent of $v$,
and $q^{-1}\sum_v\lambda_vG_q(av^2)=C_q(a)$, so the main terms give
the one in~\eqref{eq:global-gaussian}. For the errors, use $q_v\le q$
and
\[
 q\sqrt{1+4\pi^2\eta^2H^4}
 \le q+2\pi q|\eta|H^2
 \le(1+4\pi)H.
\]
It follows from~\eqref{eq:cost} that their sum is at most
\[
 \sqrt q(1+4\pi^2\eta^2H^4)^{1/4}
 \sum_{v\le H^{1/8}}v|\lambda_v|
 \ll H^{1/2}H^{1/4}=H^{3/4}.
\]
This error absorbs the $O(H^{1/4})$ error from restoring $x=0$
and removing the cutoff $x^2v^2\le N$, proving~\eqref{eq:global-gaussian}.
\end{proof}

\section{Finishing the argument}\label{sec:finishing}
We now combine Propositions~\ref{prop:arithmetic}
and~\ref{prop:gaussian} to prove Proposition~\ref{prop:uniform}, which is the uniform lower bound for $\mathcal B_H(\theta)$.

\begin{proof}[Proof of Proposition~\ref{prop:uniform}]
Taking $a=0$, $q=1$, and $\eta=0$ in
Proposition~\ref{prop:gaussian}, and using
$C_1(0)=\sum_{v\le H^{1/8}}\lambda_v=b_1=1$ and
$I_H(0)=\sqrt\pi H$, proves~\eqref{eq:normalization} and its
stated consequences.

Fix $\theta\in\R$. Dirichlet approximation, with denominator
bound $\lfloor H\rfloor$, supplies coprime integers $a,q$ such that
\[
 1\le q\le H,\qquad
 \theta=\frac aq+\eta,\qquad
 |\eta|\le\frac2{qH}.
\]
Propositions~\ref{prop:arithmetic} and~\ref{prop:gaussian} give
\[
 \mathcal B_H(\theta)
 =b_D\bigl(1+O(c_0)\bigr)I_H(\eta)
 +O\big(H^{1-c/Y}+H^{3/4}\big),
\]
where $D=D(q)$. Here we used~\eqref{eq:gaussian-integral}
to bound $|I_H(\eta)|$ by $\sqrt\pi H$.

We now check that the two factors in the main term lie
in sectors whose product has positive real part.
By~\eqref{eq:gaussian-integral},
\[
 \arg I_H(\eta)
 =-\frac12\arg(1-2\pi i\eta H^2)
 \in(-\pi/4,\pi/4).
\]
Since $b_D>0$ and the implied constant in $O(c_0)$ is absolute,
choosing $c_0$ sufficiently small also gives
\[
 \arg\big(b_D\bigl(1+O(c_0)\bigr)\big)
 \in(-\pi/4,\pi/4).
\]
Their product therefore has argument in $(-\pi/2,\pi/2)$
and hence positive real part. It follows that
\[
 \Re\mathcal B_H(\theta)
 \ge -C\big(H^{1-c/Y}+H^{3/4}\big).
\]
Since $H^{3/4}\le H^{1-c/Y}$ for sufficiently large $H$,
increasing $C$ if necessary gives
\[
 \Re\mathcal B_H(\theta)\ge-CH^{1-c/Y}
 \qquad(\theta\in\R).
\]
Together with~\eqref{eq:normalization}, this proves the proposition.
\end{proof}

\appendix
\section{The improved bound for the Furstenberg--Sárközy theorem}\label{sec:corollary}
We record the standard proof that the van der Corput property implies the S\'ark\"ozy-type result promised in the introduction.
For this we use Fourier analysis on the finite cyclic group \(\Z/M\Z\),
with the convention in Section~\ref{sec:notation}.

\begin{proof}[Proof of Corollary~\ref{cor:main}]
Let \(M=2N+1\), and regard \(A\) as a subset of \(\Z/M\Z\).  We apply Theorem
\ref{thm:main} to obtain a trigonometric polynomial
\[
 P(\theta)=a_0+\sum_{1\le n\le\sqrt N}a_n\cos(2\pi n^2\theta)
\]
such that \(P(\theta)\ge0\) for every \(\theta\in\R/\Z\), and \(P(0)=1\).
Let
\[
  \widehat{1_A}(\xi)=\sum_{a\in A}e(a\xi/M),
  \qquad \xi\in\Z/M\Z.
\]
Since \(A\subseteq [N]\), \(1\le d\le N\), and \(M=2N+1\), the congruence
\(a'-a\equiv d\pmod M\) is the same as the equality \(a'-a=d\).  Hence the
assumption that $A-A$ contains no nonzero square gives
\[
 \frac1M\sum_{\xi\pmod* M}|\widehat{1_A}(\xi)|^2P(\xi/M)
 =\frac{a_0}{M}\sum_{\xi\pmod* M}|\widehat{1_A}(\xi)|^2
 =a_0|A|.
\]
Taking real parts and separating the zero frequency, we get
\[
 a_0|A|
 =\frac{|A|^2}{M}
 +\frac1M\sum_{\substack{\xi\pmod* M\\\xi\ne0}}
 |\widehat{1_A}(\xi)|^2P(\xi/M)
 \ge\frac{|A|^2}{M},
\]
since $P(0)=1$ and every remaining summand is nonnegative.
This implies 
\[|A|\le a_0 M\ll N\exp(-c\sqrt{\log N\log\log N}),\] as desired.
\end{proof}

\section{Proof of the lower bound}

Ruzsa's dual formulation~\cite{RuzsaConnections} suggests the following
strategy for proving a lower bound for $a_0$; see also
\cite[Chapter~2, Theorem~3]{MontgomeryTenLectures}.
A probability measure $\mu$ on $\T$ satisfying
\begin{equation}\label{squareints}
 \int_{\T}e(n^2\alpha)\,d\mu(\alpha)=0
 \qquad(1\le n\le\sqrt N)
\end{equation}
would give, for every polynomial $P$ in Theorem~\ref{thm:lower-bound},
\[
 a_0=\int_{\T}P(\alpha)\,d\mu(\alpha)
 \ge\mu(\{0\})P(0)=\mu(\{0\}),
\]
since $P\ge0$ and $P(0)=1$, so it would suffice to find such a measure with $\mu(\{0\})$ large. Inspired by this strategy, we construct a family of probability measures
for which the integrals in~\eqref{squareints} can be made close to zero,
rather than requiring them to vanish exactly.

We first offer a heuristic for the proof in terms of \emph{Poisson random walks}, and afterwards we provide a short, rigorous argument; see \cite[Section~1.3]{Applebaum2009} for more background on random walks. Fix a probability distribution $\rho$ on $\T$. For $s>0$,
there are two separate random choices. First, choose the number of steps $R$ with probabilities
\[
 \Pr(R=r)=e^{-s}\frac{s^r}{r!}
 \qquad(r=0,1,2,\ldots).
\]
This is a Poisson distribution with $s$ being the expected number of steps. Then, given $R=r$, take exactly $r$ jumps, with displacements
$X_1,\ldots,X_r$ chosen independently according to $\rho$. (For this reason we call $\rho$ the \emph{jump distribution}.)
Starting at zero, the final position is
\[
 X_1+\cdots+X_R\pmod1,
\]
with the convention that the position is zero when $R=0$. 

Write $\eta_s$ for the distribution of this final position.
For exactly $r$ steps, the distribution is the convolution of $r$
copies of $\rho$. Averaging over the possible numbers of steps gives
\[
 \eta_s
 =e^{-s}\delta_0
  +e^{-s}\sum_{r=1}^{\infty}\frac{s^r}{r!}
   \underbrace{\rho*\cdots*\rho}_{r\text{ factors}},
\]
where $\delta_0$ is the probability measure concentrated at zero.
To obtain exponential decay, we choose a distribution $\rho$ satisfying
\[
 \int_{\T}e(n^2\alpha)\,d\rho(\alpha)=1-c_n,
 \qquad c_n>0.
\]
For $R=r$, independence makes the corresponding integral equal to
$(1-c_n)^r$. Thus
\[
 \int_{\T}e(n^2\alpha)\,d\eta_s(\alpha)
 =e^{-s}\sum_{r=0}^{\infty}\frac{s^r}{r!}(1-c_n)^r
 =e^{-c_ns}.
\]
The event $R=0$ has probability
$e^{-s}$ and leaves the position exactly at zero, so
$\eta_s(\{0\})\ge e^{-s}$. The measures $\mu_{p,s}$ constructed below
are distributions of walks of this kind.

Next, take $K$ independent compound Poisson walks of the kind just
described, with possibly different jump distributions
$\rho_1,\ldots,\rho_K$. In the proof below, we use one walk for each prime in
a suitably chosen set. Let $\eta_{j,s}$ be the distribution of the final position of the
$j$th walk, and write
\[
 \int_{\T}e(n^2\alpha)\,d\rho_j(\alpha)=1-c_{j,n},
 \qquad c_{j,n}>0.
\]
The calculation above gives
\[
 \int_{\T}e(n^2\alpha)\,d\eta_{j,s}(\alpha)=e^{-c_{j,n}s}.
\]
Now add the final positions of these $K$ walks modulo one.
The distribution of their sum is
\[
 \mu_s=
 \underbrace{\eta_{1,s}*\cdots*\eta_{K,s}}_{K\text{ factors}}.
\]
Independence makes the integrals multiply, giving
\[
 \int_{\T}e(n^2\alpha)\,d\mu_s(\alpha)
 =\prod_{j=1}^{K}e^{-c_{j,n}s}
 =e^{-M(n)s},
 \qquad
 M(n):=\sum_{j=1}^{K}c_{j,n}.
\]
We choose the jump distributions so that $M(n)$ has a lower bound
valid for every $1\le n\le\sqrt N$. This allows us to make all these
integrals small at the same value of $s$.

For an admissible polynomial $P$, consider its average
\begin{equation}\label{F(s)formula}
 F(s):=\int_{\T}P(\alpha)\,d\mu_s(\alpha)
 =a_0+\sum_{1\le n\le\sqrt N}a_ne^{-M(n)s}.
\end{equation}
We compare $F$ at two values of the parameter $s$: a small value
$s_0$ where we have a fixed positive lower bound, and a larger value
$s_1$ where $F$ is close to $a_0$.
Choose $s_0:=1/K$. At this value, the probability that all $K$ walks
take zero steps is $e^{-Ks_0}=e^{-1}$, leaving their sum at zero.
Since $P\ge0$ and $P(0)=1$, this gives
\[
 F(s_0)\ge e^{-1}.
\]
Next, choose $s_1\ge s_0$ so that
\[
 \sum_{1\le n\le\sqrt N}|a_n|e^{-M(n)s_1}\le a_0.
\]
The estimates in the proof show that such a choice is possible.
Consequently, $|F(s_1)-a_0|\le a_0,$
and hence $F(s_1)\le2a_0.$
We will use Harnack's inequality
to bound $F(s_1)$ below in terms of $F(s_0)$. This transfers the
known lower bound for $F(s_0)$ to $F(s_1)$, where the inequality
$F(s_1)\le2a_0$ then forces a lower bound for $a_0$.

The finite sum in~\eqref{F(s)formula} extends to an entire function $F(w)$.
The key point is that for
$w$ in a sector $|\arg w|<\pi/(2\kappa)$, with $\kappa\ge1$, we can
still construct probability measures $\mu_w$ such that
\[
 \Re F(w)=\int_{\T}P(\alpha)\,d\mu_w(\alpha)\ge0.
\]
Thus $\Re F$ is a nonnegative harmonic function on the sector.
Harnack's inequality compares its values at $s_0$ and $s_1$, giving
\[
 2a_0\ge F(s_1)\ge\left(\frac{s_0}{s_1}\right)^\kappa F(s_0)
 \ge e^{-1}(Ks_1)^{-\kappa}.
\]

We use the following form of Harnack's inequality
for harmonic functions \cite[Lemma~2.5]{BearSmith2006}.

\begin{lemma}[Harnack's inequality]\label{lem:lower-harnack}
If $h$ is a nonnegative harmonic function defined on the right
half-plane, then for all real numbers $0<s_0\le s_1$,
\[
 h(s_1)\ge\frac{s_0}{s_1}\,h(s_0).
\]
\end{lemma}

A quick corollary of this is the following version of Harnack's inequality for harmonic
functions on a sector contained in the right half-plane.

\begin{corollary}\label{cor:lower-harnack}
Let $\kappa\ge1$. If $h$ is a nonnegative harmonic function on the sector
\[
 \{w\in\C\setminus\{0\}:\ |\arg w|<\pi/(2\kappa)\},
\]
then, for all real numbers $0<s_0\le s_1$,
\begin{equation}
 h(s_1)\ge\left(\frac{s_0}{s_1}\right)^\kappa h(s_0).
\end{equation}
\end{corollary}

\begin{proof}
This follows immediately from Lemma~\ref{lem:lower-harnack} applied
to the nonnegative harmonic function
$z\mapsto h(z^{1/\kappa})$ on the right half-plane at the positive
real points $s_0^\kappa\le s_1^\kappa$.
\end{proof}

For a prime $p$, let $\nu_p(n)$ denote the $p$-adic valuation
of a positive integer $n$. In other words, $\nu_p(n)$ is the largest
nonnegative integer $k$ such that $p^k\mid n$.

\begin{lemma}\label{lem:lower-local}
Let $p\equiv3\pmod4$ be prime. If $\Re w>0$ and
$|\Im w|\le(\Re w)/(2\sqrt p)$, there is a probability measure
$\mu_{p,w}$ on $\T$ such that
\begin{equation}
 \int_{\T}e(n^2\alpha)\,d\mu_{p,w}(\alpha)
 =\exp\!\left(-2^{-\nu_p(n)}w\right)
 \qquad(n\ge1),
\end{equation}
and
\[
 \mu_{p,w}(\{0\})\ge e^{-\Re w}.
\]
\end{lemma}
\begin{proof}
Write $s=\Re w>0$ and put
\[
 \tau=-\frac{2i\Im w}{s},
 \qquad |\tau|\le\frac1{\sqrt p}.
\]
For each $a\in\Z/p\Z$, define
\[
 q_a=\frac1p+\frac{\Im\tau}{\sqrt p}\left(\frac ap\right),
\]
where $(\frac ap)$ denotes the Legendre symbol.
These coefficients are nonnegative and sum to one. For each $j\ge0$,
let
\[
 \sigma_{j,\tau}
 =\frac1{p^{2j}}\sum_{b\in\Z/p^{2j}\Z}\,\sum_{a\in\Z/p\Z}
 q_a\delta_{(a+pb)/p^{2j+1}},
\]
where $\delta_x$ denotes the Dirac measure at $x\in\T$.
This is a finitely supported probability measure. Orthogonality gives
\[
 \frac1{p^{2j}}\sum_{b\in\Z/p^{2j}\Z}
 e\!\left(\frac{n^2b}{p^{2j}}\right)
 =\mathbf1_{p^j\mid n},
\]
so
\[
 \int_{\T}e(n^2\alpha)\,d\sigma_{j,\tau}(\alpha)=0
\]
when $p^j\nmid n$. If $n=p^jr$ with $p\nmid r$, the quadratic Gauss sum
\[
 \sum_{a\in\Z/p\Z}\left(\frac ap\right)e\!\left(\frac{ar^2}{p}\right)
 =\sum_{x\in \Z/p\Z}e\!\left(\frac{x^2}{p}\right)=i\sqrt p
\]
implies
\[
 \int_{\T}e(n^2\alpha)\,d\sigma_{j,\tau}(\alpha)
 =\sum_{a\in\Z/p\Z}q_a e\!\left(\frac{ar^2}{p}\right)
 =i\Im\tau=\tau.
\]
If $p\mid r$, every exponential in this sum is one, so that
\begin{equation}
 \int_{\T}e(n^2\alpha)\,d\sigma_{j,\tau}(\alpha)
 =\begin{cases}
  0,&\nu_p(n)<j,\\
  \tau,&\nu_p(n)=j,\\
  1,&\nu_p(n)>j.
 \end{cases}
\end{equation}

Now consider the probability measure
\[
 \rho=\sum_{j=0}^{\infty}2^{-j-1}\sigma_{j,\tau}.
\]
For $k=\nu_p(n)$, we obtain
\[
 \int_{\T}e(n^2\alpha)\,d\rho(\alpha)
 =\sum_{j=0}^{k-1}2^{-j-1}+2^{-k-1}\tau
 =1-2^{-k}\frac ws.
\]
Define
\[
 \mu_{p,w}=e^{-s}\delta_0
 +e^{-s}\sum_{r=1}^{\infty}\frac{s^r}{r!}
 \underbrace{\rho*\cdots*\rho}_{r\text{ factors}}.
\]
Then
$\mu_{p,w}$ is a probability measure. Moreover,
\[
 \int_{\T}e(n^2\alpha)\,d\mu_{p,w}(\alpha)
 =e^{-s}\sum_{r=0}^{\infty}\frac{s^r}{r!}
   \left(1-2^{-k}\frac ws\right)^r
 =\exp\!\left(-s+s\left(1-2^{-k}\frac ws\right)\right)
 =e^{-2^{-k}w}.
\]
Finally, the term $e^{-s}\delta_0$ gives $\mu_{p,w}(\{0\})\ge e^{-s}$.
\end{proof}

We now deduce Theorem~\ref{thm:lower-bound} from this lemma.
\begin{proof}[Proof of Theorem~\ref{thm:lower-bound}]
Put
$X=\lfloor\sqrt N\rfloor$ and fix a polynomial
\[
 P(\alpha)=a_0+\sum_{1\le n\le X}a_n\cos(2\pi n^2\alpha),
 \qquad P\ge0,\qquad P(0)=1.
\]
Orthogonality and nonnegativity give
\begin{equation}\label{eq:lower-coefficients}
 a_0=\int_{\T}P(\alpha)\,d\alpha>0,
 \qquad
 |a_n|=2\left|\int_{\T}P(\alpha)e(-n^2\alpha)\,d\alpha\right|
 \le2a_0.
\end{equation}
We may assume that $X$ is sufficiently large and $a_0\le1$. Let
\[
 y=\frac{4\log X}{\log\log X},
\]
and denote by $\mathcal P$ the set of primes $p\le y$ such that
$p\equiv 3\pmod4$. The prime number theorem for primes in the fixed
progression $3\pmod4$ gives
\[
 \sum_{p\in\mathcal P}\log p=(1+o(1))\frac{y}{2}\ge\frac y4
 =\frac{\log X}{\log\log X}.
\]
For every $1\le n\le X$,
\[
 \left(\min_{p\in\mathcal P}\nu_p(n)\right)
 \sum_{p\in\mathcal P}\log p
 \le\sum_{p\in\mathcal P}\nu_p(n)\log p
 \le\log n\le\log X.
\]
Thus some $p\in\mathcal P$ has $\nu_p(n)\le\log\log X$, and hence
\begin{equation}\label{eq:lower-exponent-bound}
 M(n):=\sum_{p\in\mathcal P}2^{-\nu_p(n)}
 \ge2^{-\log\log X}\ge\frac1{\log X}.
\end{equation}

For every $p\in\mathcal P$ and every $w\in\C$ with $\Re w>0$ and
$|\Im w|\le(\Re w)/(2\sqrt y)$,
Lemma~\ref{lem:lower-local} supplies a probability measure
$\mu_{p,w}$ on $\T$. Set
\[
 \mu_w:=\mathop{*}_{p\in\mathcal P}\mu_{p,w},
\]
the convolution of all $\mu_{p,w}$ over $p\in\mathcal P$. Then
\begin{equation}
 \int_{\T}e(n^2\alpha)\,d\mu_w(\alpha)
 =\prod_{p\in\mathcal P}\exp\!\left(-2^{-\nu_p(n)}w\right)
 =e^{-M(n)w}.
\end{equation}
Consequently, the entire function
\[
 F(w):=a_0+\sum_{1\le n\le X}a_ne^{-M(n)w}
\]
satisfies
\begin{equation}\label{eq:lower-average-positive}
 \Re F(w)=\int_{\T}P(\alpha)\,d\mu_w(\alpha)\ge0
\end{equation}
for every $w\in\C$ with $\Re w>0$ and
$|\Im w|\le(\Re w)/(2\sqrt y)$. For any real $s>0$,
\begin{equation}\label{eq:lower-real-axis}
 F(s)\ge\mu_s(\{0\})P(0)
 \ge\prod_{p\in\mathcal P}\mu_{p,s}(\{0\})
 \ge\prod_{p\in\mathcal P}e^{-s}
 = e^{-s|\mathcal P|}.
\end{equation}

Put $s_0=1/|\mathcal P|$ and $s_1=2(\log X)^2$. By~\eqref{eq:lower-coefficients}
and~\eqref{eq:lower-exponent-bound},
\[
 |F(s_1)-a_0|
 \le2a_0X e^{-s_1/\log X}
 =\frac{2a_0}{X}\le a_0,
\]
so $F(s_1)\le2a_0$. Set
\[
 \kappa=\frac{\pi}{2\arctan(1/(2\sqrt y))}\asymp\sqrt y.
\]
The interior of the sector in~\eqref{eq:lower-average-positive} is
$|\arg w|<\pi/(2\kappa)$. Applying Corollary~\ref{cor:lower-harnack}
to $\Re F$ at $s_0$ and $s_1$, and then
using~\eqref{eq:lower-real-axis}, we have
\begin{equation}
 2a_0\ge F(s_1)
 \ge(s_1|\mathcal P|)^{-\kappa}F(s_0)
 \ge e^{-1}(s_1|\mathcal P|)^{-\kappa}.
\end{equation}
Since $|\mathcal P|\le y$ and $\log(s_1 y)\ll\log\log X$, it follows that
\[
 \log\frac1{a_0}
 \ll\sqrt y\,\log\log X
 \ll\sqrt{\log X\log\log X}
 \ll\sqrt{\log N\log\log N},
\]
which proves the theorem.
\end{proof}

\bibliographystyle{alpha}
\bibliography{references}
\end{document}